\documentclass[a4paper, 11pt]{amsart}
\usepackage{amsmath,amsthm,amssymb,amscd,amsfonts}

\usepackage[backref=page]{hyperref}

\newtheorem{theorem}{Theorem}[section]
\newtheorem{proposition}[theorem]{Proposition}
\newtheorem{lemma}[theorem]{Lemma}
\newtheorem{corollary}[theorem]{Corollary}
\theoremstyle{remark}
\newtheorem*{remark}{Remark}

\def\MU{\mathit{MU}}
\def\Z{\mathbb Z}
\def\C{\mathbb C}
\def\P{\mathcal P}

\begin{document}

\title{Landweber exactness of the formal group law in $c_1$-spherical bordism}

\author{G. S. Chernykh}
\date{}
\address{Faculty of Mathematics and Mechanics, Moscow
State University, Russia;\newline
Steklov Mathematical Institute of the Russian Academy of Sciences, Moscow, Russia;\newline
National Research University Higher School of Economics, Russia}
\email{aaa057721@gmail.com}

\thanks{This work was performed at the Steklov International Mathematical Center and supported by the Ministry of Science and Higher Education of the Russian Federation (agreement no. 075-15-2022-265). The author is supported by Theoretical Physics and Mathematics Advancement Foundation ``BASIS''. The article was prepared within the framework of the HSE University Basic Research Program.}

\maketitle

\begin{abstract}
We describe the structure of the coefficient ring $W^*(pt)=\varOmega_W^*$ of the $c_1$-spherical bordism theory for an arbitrary $SU$-bilinear multiplication. We prove that for any $SU$-bilinear multiplication the formal group of the theory $W^*$ is Landweber exact. Also we show that after inverting the set $\P$ of Fermat primes there exists a complex orientation of the localized theory $W^*[\P^{-1}]$ such that the coefficients of the corresponding formal group law generate the whole coefficient ring $\varOmega_W^*[\P^{-1}]$.
\end{abstract}

\section*{Introduction}

Special unitary bordism, or $SU$-bordism, is the bordism theory of stably complex manifolds with a special unitary structure, where an $SU$-structure on a manifold $M$ is defined by a reduction of the structure group of the stable tangent bundle to the group $SU(N)$ (or, equivalently, by a trivialization of the one-dimensional complex bundle $\det TM$). Details of the construction of the $SU$-bordism theory and a description of its coefficient ring $\varOmega^{SU}_*$ can be found in ~\cite{novi67,ston68,c-l-p19}.

All known calculations of the coefficient ring $\varOmega^{SU}_*$ involve the $c_1$-spherical bordism groups $\varOmega^W_*$, which sit between $\varOmega^{SU}_*$ and complex bordism groups $\varOmega^U_*$ (of manifolds with a stably complex structure).

Originally, the groups $\varOmega_*^W$ appeared in the work of Conner and Floyd \cite{co-fl66} as the kernel of a certain operation $\Delta$ in complex bordism (see also \cite{ston68, c-l-p19}), and exactly in this form the groups $\varOmega_*^W$ arise in the calculation of the special unitary bordism groups $\varOmega_*^{SU}$ via the Adams--Novikov spectral sequence \cite{novi67, c-l-p19}.

The groups $\varOmega^W_*$ are the coefficient groups of the $c_1$-spherical bordism theory $W_*$ introduced by Stong~\cite{ston68}. Namely, a $c_1$-spherical stably complex structure on a manifold $M$ is given by a reduction of the classifying map $M\to\C P^{\infty}$ of the bundle $\det TM$ to a map to the ``sphere'' $\C P^1\subset\C P^{\infty}$. (Note that a stably complex structure is an $SU$-structure if and only if the bundle $\det T M$ is trivial, that is, its classifying map is reduced to a map to the point ``$\C P^0$''.) Details of the definition and properties of the theory $W_*$ can be found in ~\cite{ston68,c-l-p19, c-p23}.

Complex orientations of the theory $W^*$ and the corresponding formal group laws were introduced in the work of V. M. Buchstaber~\cite{buch72}, and subsequently studied in detail in the work~\cite{c-p23}. The theory $W^*$ is a module over the $SU$-bordism theory, and all $SU$-bilinear multiplications on $W^*$ were described in~\cite{c-p23}. In this paper we describe the structure of the coefficient ring $W^*(pt)=\varOmega_W^*$ for an arbitrary $SU$-bilinear multiplication (Theorem \ref{Wkring}), prove the Landweber exactness of the formal group law in the theory of $W^*$ for an arbitrary $SU$-bilinear multiplication (Theorem \ref{Wland}), and also show that after inverting of Fermat primes there exists a complex orientation for which the coefficients of the formal group law generate the whole localized ring $\varOmega_W^*$. The latter result was stated in \cite{buch72}. We give a proof of this fact (see Theorem \ref{Fermat}).

The author is deeply grateful to Taras Panov for suggesting the problem, fruitful discussions and constant attention to the work.

\section{Preliminaries}

Denote by $m_k$ the greatest common divisor of the binomial coefficients ${k+1\choose i}$, $1 \le i \le k$. Then the following well known equality holds

\[
m_k =
\begin{cases}
p&\text{if $k+1=p^s$ for a prime $p$,}\\
1&\text{in other cases.}
\end{cases}
\]

Recall \cite{novi62} that the coefficient ring of complex cobordism has the form $\varOmega_U^*=\Z[a_1, a_2, \ldots]$, $|a_k|=-2k$, and the polynomial generators $a_i$ are characterized by the property that their higher $s$-numbers are given by $s_k(a_k)=\pm m_k$. We denote by $J$ the ideal of elements of nonzero degree in the ring $\varOmega_U^*$. Thus, $J^2$ is the ideal of decomposables.

The standard complex orientation of the complex cobordism theory $\MU^*$ corresponds to the formal group law $F(u,v) = u+v+\sum\alpha_{ij}u^i v^j$ in complex cobordism. This formal group law is the universal formal group law in the sense of Lazard\cite{quil69}, that is, for any formal group law $F_R$ over any ring $R$ there exists a unique ring homomorphism $\psi\colon\varOmega^*_U\to R$ mapping the formal group law $F$ to $F_R$. In this case we say that the homomorphism $\psi$ classifies the formal group law $F_R$.

\begin{lemma}\label{simple}
Let $F_R(u, v)$ be a graded formal group law over a non-positively graded ring $R$. Then there exist elements $r_k \in R^{-2k}$ such that 
$$F_R(u, v)=u+v+\sum\limits_{k\ge1}r_k\frac{((u+v)^{k+1}-u^{k+1}-v^{k+1})}{m_k} \mod J_R^2,$$
where $J_R =R^{<0}$ and $J_R^2$ is the ideal of decomposables.

The $n$-th power series corresponding to a formal group law $F_R(u,v)$ satisfies the equality 
\[
  [n]_R(u) =nu+\sum r_k \frac{(nu)^{k+1}-nu^{k+1}}{m_k} \mod {J_R^2}.
\] 
\end{lemma}

\begin{proof}
There is the following formula for the universal formal group law modulo decomposables:
\[
  F(u, v) = u+v+\sum_{k \ge 1} a_k \frac{(u+v)^{k+1} - u^{k+1} - v^{k+1}}{m_k} \mod {J^2},
\]
where $a_k$ is a polynomial generator of the ring $\varOmega^*_U$ with the $s$-number $s_k(a_k)=-m_k$ (see~\cite{adam74}). Then for the $n$th power we get
\[
  [n](u) = nu + \sum_{k \ge 1} a_k \frac{(nu)^{k+1} - nu^{k+1}}{m_k} \mod{J^2}.
\]
The required formulas are obtained by applying the homomorphism $\psi\colon\varOmega^*_U\to R^*$ classifying~$F_R$ to the above and observing that $\psi(J^2) \subset J_R^2$. Note that $r_k=\psi(a_k)$.
\end{proof}

There is a natural inclusion $W^*\hookrightarrow\MU^*$ of the $c_1$-spherical cobordism theory $W^*$ into complex cobordism. Furthermore, $W^*$ splits as a direct summand by means of natural projections, for example, the classical Stong projection $\pi_0\colon\MU^*\to W^*$. On the coefficients of the theories, the Stong projection sends the bordism class of a stably complex manifold $M$ to the bordism class of the submanifold $N$ in $M\times\C P^1$ dual to the one-dimensional complex bundle $\det T M\otimes\mathcal O(1)$. Using the Stong projection, one can define a multiplication on $W^*$ by the formula $a*b = \pi_0(ab)$ (note that the subgroup $W^*$ is not closed under the standard product $ab$ in $\MU^*$). We have
\begin{equation}\label{mult}
a*b=ab+2\alpha_{12}\partial a\partial b,
\end{equation} where $\alpha_{12}$ is the corresponding coefficient of the formal group law in complex cobordism, and $\partial$ is the operation in complex cobordism acting on the coefficients by sending the bordism class of a manifold $M$ to the bordism class of the submanifold $N$ dual to $\det TM$. For details, see \cite{ston68}, \cite{c-p23}.

With respect to the multiplication $a*b=\pi_0(ab)$ defined by the Stong projection $\pi_0$, the ring of coefficients of the theory $W^*$ has the form
\begin{equation}\label{W}
\varOmega_W^*=\Z[x_1, x_{k \ge 3}], \, |x_k|=-2k,
\end{equation} where the polynomial generators are characterized by the property that $s_k(x_k)=\pm m_k m_{k-1}$, and can be chosen such that $\partial(x_{2i})=x_{2i-1}$ for $i>1$ (and hence $\partial(x_{2i-1})=0$) \cite[Chapter X]{ston68}, \cite[Theorem 6.10]{c-l-p19}. Note that $x_1=a_1 = \pm [\C P^1]$.

We have the forgetful homomorphism from $SU$-bordism to $c_1$-spherical bordism $\varOmega^*_{SU}\to \varOmega^*_W$. After inverting $2$ this homomorphism becomes injective and, moreover,
\begin{equation}\label{W1/2}
\varOmega^*_W[1/2]=\varOmega^*_{SU}[1/2] \oplus x_1 \varOmega^*_{SU}[1/2].
\end{equation}
Furthermore,
\begin{equation}\label{SU1/2}
\varOmega^*_{SU}[1/2]=\Z[1/2][y_2, y_3, \ldots].
\end{equation}
Here one can take $y_2=x_1 *x_1$  and $y_i=x_i - \frac{1}{2}x_1\partial(x_{i})$ for $i>2$ (see \cite{ston68}, \cite{c-l-p19}).
As a result, we get
\begin{equation*}\label{W1/2p}
\varOmega^*_W[1/2]=\Z[1/2][x_1, y_2, y_3, \ldots]/(x_1 * x_1 = y_2).
\end{equation*}

Recall that for a given (graded) formal group law $F_R$ over a graded ring $R_*$ and a fixed prime number $p$, the coefficient for $u^{p^n}$ in the power series $[p]_{F_R}(u)$ is traditionally denoted by $v_n$ (this notation does not reflect the prime~$p$). Then we have the following

\begin{theorem}[Landweber Exact Functor Theorem {\cite{lan76}}]\label{Land}
Suppose that for any prime number $p$ the sequence $(p, v_1, v_2, \ldots)$ is regular in the ring $R_*$, that is, for any $n$ multiplication by the element $v_n$ is injective in the quotient ring $R_*/(p, v_1, \ldots, v_{n-1})$ (including the multiplication by $p$ in~$R_*$). Then the functor $\MU_*(-)\otimes_{\varOmega^U_*}R_*$ is exact, so it defines a homology theory. Here the structure of an $\varOmega^U_*$-module on $R_*$ is given by the ring homomorphism $\varOmega^U_*\to R_*$ classifying the given formal group law ~$F_R$. Similarly, the formula $\MU^*(-)\otimes_{\varOmega_U^*}R^*$ defines a ring cohomology theory on finite cell complexes.
\end{theorem}

A formal group law $F_R$ satisfying the conditions of the theorem above is called \emph{Landweber exact}.

\begin{corollary}
If the formal group law corresponding to a complex-oriented theory $h^*$ is Landweber exact, then there is a natural isomorphism of multiplicative cohomology theories $h^*(-) \simeq \MU^*(-) \otimes_{\varOmega_U^*} h^*(pt)$ on finite cell complexes, and an isomorphism of the corresponding homology theories.
\end{corollary}

The following fact is well known.

\begin{proposition}
Let $F_1$ and $F_2$ be two formal group laws over a ring $R$, and let $g \colon F_1 \to F_2$ be an isomorphism between them given by a power series $g(u) = \lambda u + \cdots \in R[[u]],$ where $\lambda \in R^\times$, $F_2(u, v) = g (F_1 (g^{-1}(u), g^{-1}(v)))$ and $[ p]_2(u) = g([p]_1(g^{-1}(u)))$. Denote by $v_n$ and $v_n'$ the coefficients for $u^{p^n}$ in the series $[p]_1(u)$ and $[p]_2(u)$, respectively.

Then $v'_n = s v_n \mod{(p, v_1, \ldots, v_{n-1})}, \, s \in R^\times$. In particular, the ideal $(p, v_1, \ldots, v_n)$ depends only on the isomorphism class of the formal group law.
\end{proposition}

\begin{proof}
The proof is by induction on $n$. If $n=0$, the ideal in question is $(p)$ and the statement is trivial. Let us prove that the equality $(p, v_1, \ldots, v_{n-1}) = (p, v'_1, \ldots, v'_{n-1})$ implies that $v'_n = s v_n \mod{(p, v_1, \ldots, v_{n-1})}, \, s \in R^\times$.

It is known that if for a formal group law $F_R$ the equalities $0=p=\cdots=v_{n-1}$ hold in~$R$, then $[p]_R(u) = \varphi(u^{p ^n}),\,$ where $\varphi(u) = v_n u + \cdots$ is a power series in $u$ (see, for example, \cite[Lemma A2.2.6]{rav86}).

Therefore, reducing the formal group laws modulo $(p, v_1, \ldots, v_{n-1})$ we obtain $[p]_1(u) = v_n u^{p^n} + \cdots$ and $[p]_2(u) = v'_n u^{p^n} + \cdots$. On the other hand, $[p]_2(u) = g([p]_1(g^{-1}(u)))=\lambda \lambda^{-p^n} v_n u^{p^n} + \cdots$. Thus, $v'_n = \lambda^{1-p^n} v_n \mod {(p, v_1, \ldots, v_{n-1})}$.
\end{proof}

\begin{corollary}
The Landweber exactness of the formal group law corresponding to a complex-oriented theory does not depend on a complex orientation.
\end{corollary}

\section{Main results}

The theory $W^*$ is a module over the $SU$-bordism theory $MSU^*$. By~\cite[Theorem 2.21]{c-p23}, an arbitrary $SU$-bilinear multiplication on $W^*$ has the form
\begin{equation}\label{mult2}
a\tilde * b = ab + (2 \alpha_{12} + w) \partial a \partial b,
\end{equation} where $w$ is an arbitrary element of $\varOmega^{-4}_W$.
For the standard multiplication $a*b = \pi_0(ab)$ given by the Stong projection $\pi_0$, we have $w=0$. If the multiplication $a\tilde*b=\pi(ab)$ is defined by an arbitrary $SU$-linear projection $\pi \colon \MU^* \to W^*$, then the element $w$ is divided by $2$.

By \eqref{W} we have $\varOmega^{-4}_W = \Z \langle x_1 * x_1 \rangle$, hence  the element $w$ in~\eqref{mult2} has the form $q x_1 * x_1$, $q \in \Z$. Then an $SU$-bilinear multiplication takes the form
\begin{equation}\label{multw}
a *_q b = ab + (2 \alpha_{12} + q x_1 * x_1)\partial a \partial b = a*b + q x_1 * x_1 \partial a \partial b.
\end{equation}
In this notation, the standard multiplication is $a *_0 b$.

\begin{theorem}\label{Wkring}
With respect to the multiplication $*_q$, the ring $\varOmega^*_W$ has the form 
\begin{equation*}\label{tW1/2}
(\varOmega^*_W, *_q) = \Z[x_1, x_2, x_3, \ldots]/(x_1*_qx_1=(4q+1)x_2).
\end{equation*}
In particular, for any $SU$-bilinear multiplication of $a *_q b$, except for the standard one $a * b = ab + 2 \alpha_{12} \partial a \partial b$ (defined by the Stong projection $\pi_0$), the ring $\varOmega^*_W$ is not polynomial.
\end{theorem}

\begin{proof}

Consider the elements $x_i \in \varOmega^*_W, \, i \ne 2$ from \eqref{W} with $x_1=[\C P^1]$ and set $x_2 = x_1 * x_1$. Then~\eqref{multw} implies $x_1 *_q x_1 = x_1 * x_1 + 4q x_1 * x_1 = (4q+1)x_2$, since $\partial x_1 = 2$. Therefore, we have a ring homomorphism
\begin{equation}\label{homo}
\varphi \colon \Z[x_1, x_2, x_3, \ldots]/(x_1*_qx_1=(4 q+1)x_2) \to (\varOmega^*_W, *_q).
\end{equation}
We claim that this is an isomorphism. To see this, first note that  there is an isomorphism of abelian groups
$$ \Z[x_1, x_2, x_3, \ldots]/(x_1*_qx_1=(4q+1)x_2) = \Z[x_2, x_3, \ldots] \oplus x_1 \Z [x_2, x_3, \ldots]$$

Since the rings in \eqref{homo} are graded, and the graded components are free abelian group of finite rank, it suffices to prove that $\varphi\otimes \Z/(2)$ and $\varphi\otimes \Z[1/2]$ are isomorphisms.

\smallskip
\textit{1) $\varphi \otimes \Z/(2)$ is an isomorphism}
\smallskip

By reducing modulo $2$ we obtain 
$$
  \Z/(2)[x_1, x_2, x_3, \ldots]/(x_1*_qx_1=(4q+1)x_2) = \Z/(2)[x_1, x_3, x_4, \ldots]
$$ 
Consider the filtrations $$ \Z/(2)[x_1, x_3, x_4, \ldots] \supset (x_1) \supset (x_1^2) \supset \cdots$$ 
and 
$$(\varOmega^*_W/(2), *_q) \supset x_1 *_q \varOmega^*_W/(2) \supset x_1 *_q x_1 *_q \varOmega^*_W/(2) \supset \cdots$$
The homomorphism $\varphi$ preserves these filtrations, and therefore it is sufficient to prove that it induces an isomorphism on the associated quotients.

The quotient $(x_1^n)/(x_1^{n+1})$ has a basis consisting of monomials $x_1^n x_3^{i_3} x_4^{i_4} \cdots$.

Since $\partial(x_1)=2$, it follows from~\eqref{multw} that $x_1*_q a=x_1*a$ in the ring $(\varOmega^*_W/(2), *_q)$. Therefore, the ideal in $\varOmega^*_W$ generated by the $n$th power $x_1^{*_q n}$ with respect to the multiplication $*_q$ coincides with the ideal generated by the $n$th power $x_1^{* n}$ with respect to the standard multiplication. By virtue of formula~\eqref{W} the quotient $x_1^{*_q (n+1)} *_q \varOmega^*_W / x_1^{*_q n} *_q \varOmega^*_W = x_1^{* (n+1)} * \varOmega^*_W / x_1^{* n} * \varOmega^*_W$ has a basis consisting of monomials $x_1 ^{*n} * x_3 ^{* i_3} * x_4^{*i_4}\cdots$. In this quotient, the identity $x_1^{*_q n} *_q x_3 ^{*_q i_3} *_q x_4^{*_q i_4}\cdots = x_1^{*n} * x_3 ^{* i_3} * x_4^{ *i_4}\cdots$ holds, by formula~\eqref{multw}. Since the homomorphism $\varphi$ takes $x_1^n x_3^{i_3} x_4^{i_4}\cdots$ to $x_1^{*_q n} *_q x_3 ^{*_q i_3} *_q x_4^{*_q i_4}\cdots$, it is an isomorphism on the associated quotients of filtrations.

\smallskip
\textit{2) $\varphi \otimes \Z[1/2]$ is an isomorphism}
\smallskip

Formulas \eqref{W1/2} and \eqref{SU1/2} imply that $$\varOmega^*_W[1/2] = \Z[1/2][y_2, y_3, \ldots] \oplus x_1 \Z[1/2][y_2, y_3, \ldots],$$
where $\Z[1/2][y_2, y_3, \ldots] = \varOmega^*_{SU}[1/2]$, $y_2 = x_1 * x_1$, $y_{2i-1} = x_{2i-1}$ and $y_{2i} = x_{2i} - \frac{1}{2}x_1 y_{2i-1}$ for $i>1$. Furthermore, since $\partial$ vanishes on $\varOmega^*_{SU}[1/2]$, the multiplication $*_q$ is given by $$(a_1 + x_1 b_1)*_q(a_2 + x_1 b_2) = a_1 a_2 + x_1 (a_1 b_2 + b_1 a_2) + (4q+1)y_2 b_1 b_2$$ for $ a_i, b_i \in \Z[1/2][y_2, y_3, \ldots]$.
Hence, there is an isomorphism $$(\varOmega^*_W[1/2], *_q) = \Z[1/2][x_1, y_2, y_3, \ldots]/(x_1*_qx_1=(4q+1)y_2).$$

Therefore, the homomorphism $\varphi \otimes \Z[1/2]$ has the form
$$ \Z[1/2][x_1, x_2, \ldots]/(x_1*_qx_1=(4q+1)x_2) \to \Z[1/2][x_1, y_2, \ldots]/(x_1*_qx_1=(4q+1)y_2),$$
where $x_1 \mapsto x_1$, $x_2 \mapsto y_2$, $x_{2i-1} \mapsto y_{2i-1}$ and $x_{2i} \mapsto y_{2i} + \frac{1}{2}x_1 y_{2i-1}$ with $i>1$. It follows that $\varphi \otimes \Z[1/2]$ is an isomorphism.
\end{proof}

The projection $\pi_0$ defines a complex orientation $w=\pi_0(u)$ on $W^*$, where $u$ is the standard orientation of $\MU^*$. The complex orientation and multiplication $*_q$ determine the formal group law
\begin{equation*}\label{fgw}
F_W (u, v) = u + v + \sum _{i, j >0} \omega_{ij} *_q u^{*_q i} *_q v^{*_q j}.
\end{equation*}

\begin{theorem}\label{Wland}
The formal group law $F_W (u, v)$ over the ring $(\varOmega_W^*, *_q)$ is Landweber exact.
\end{theorem}

\begin{corollary}
For any $SU$-bilinear multiplication on the theory $W^*$, there is a natural isomorphism $W_*(-) = \MU_*(-)\otimes_{\varOmega^U_*}\varOmega^W_*$ and a ring isomorphism $W^*(-)=\MU^*(-)\otimes_{\varOmega_U^*}\varOmega^*_W$ for finite cellular complexes.
\end{corollary}

\begin{proof}[Proof of Theorem \ref{Wland}.] By Theorem \ref{Wkring} we have $$(\varOmega^*_W, *_q) = Z[x_1, x_2, x_3, \ldots]/(x_1*_qx_1=(4q+1)x_2).$$ Denote by $J_W$ the ideal of elements of nonzero degree in $\varOmega^*_W$ and by $J_W *_q J_W \subset \varOmega_W^*$ the ideal of decomposable elements in the ring $(\varOmega_W^*, *_q)$. By $*$ we denote the standard multiplication $*_0$ defined by the Stong projection.

If we consider the groups $\varOmega_W^*$ as subgroups in $\varOmega_U^*$, then by \cite[Lemma 3.8, Lemma 3.11]{c-p23} we have the following equation

\begin{equation*}\label{FmodJ}
F_W (u, v) = u + v + \sum \limits_{k \ge 1} w_k \frac{((u+v)^{k+1}-u^{k+1}-v^{k+1})}{m_k} \mod {J^2},
\end{equation*}
where
\begin{multline}\label{wmodJ}
w_1=-x_1, \quad w_2= -2q x_1 * x_1 \mod {J^2},\\ w_{k} = a_k (1+(-1)^k (k+1)) \mod {J^2}\quad \text{ for }k>2,
\end{multline}
and $a_k$ is a polynomial generator of the ring $\varOmega^*_U$ with $s_k(a_k)=-m_k$.

It is clear that $w_1 = \omega_{11}$. Moreover, by formula \eqref{W} we get that $\varOmega_W^{-4} = \Z\langle x_1*x_1\rangle = \Z\langle \C P^1*\C P^1\rangle$. Now~\eqref{mult} implies $\C P^1*\C P^1 = 9 (\C P^1)^2 - 8\C P^2$ in $\varOmega_U^*$. Hence, there is an indecomposable part in the generator of $\varOmega_W^{-4}$. It follows that $J^2\cap\varOmega^{-4}_W = 0$ and therefore $w_2= -2q x_1 *x_1$. Hence, $\omega_{12}=-2q x_1 * x_1=-2qx_2$.

Expanding the $n$th power in the formal group law $F_W(u,v)=u+v+\sum \omega_{ij}u^iv^j$ up to the third term we get
\begin{equation*}\label{nu}
  [n]_W(u) = nu + \frac{n(n-1)}{2} \omega_{11}u^2 + \Bigl(\frac{n(n-1)(n+1)} {3}\omega_{12}+\frac{n(n-1)(n-2)}{6}\omega_{11}^2\Bigr)u^3 + \cdots
\end{equation*}
Substituting $\omega_{11} = -x_1$ and $\omega_{12}=-2q x_2$ here we obtain
\begin{multline}\label{nuW}
[n]_W(u) = nu + \frac{n(n-1)}{2} x_1 u^2\\ + \Bigl(-2q\frac{n(n-1)(n+1)}{3}x_2+\frac{n(n-1)(n-2)}{6}x_1*_qx_1\Bigr)u^3 + \cdots \\
=nu + \frac{n(n-1)}{2} x_1 u^2 + \Bigl ((4q+1)\frac{n(n-1)(n-2)}{6} -2q\frac{n(n-1)(n+1)}{3} \Bigr) x_2 u^3 + \cdots
\end{multline}

Next we calculate the higher terms in the expansion of
$[n]_W(u)$ modulo decomposables in $(\varOmega^W_*, *_q)$. By Lemma~\ref{simple}, 
$$F_W(u, v)=u+v+\sum \limits_{k \ge 1} r_k \frac{((u+v)^{k+1}-u^{k+1}-v^{k+1})}{m_k} \mod J_W *_q J_W$$ for some $r_k \in \varOmega_W^{-2k}$.
Since in dimensions $>4$ elements from $J_W*_q J_W$ are decomposable in $\varOmega^*_U$, by formula \eqref{wmodJ} we have $s_k(r_k)=s_k(w_k)=s_k(a_k (1+(-1)^k (k+1)))=-m_k(1+(-1)^k (k+1))$ for $k>2$.
On the other hand, by Theorem \ref{Wkring} we have $r_k=\lambda_k x_k \mod J_W*_qJ_W$ for $\lambda_k \in \Z$. Hence, $-m_k(1+(-1)^k (k+1))=s_k(r_k)=\lambda_k s_k(x_k)=\lambda_k m_k m_{k-1}$. As a result, we obtain that $r_k = -\frac{1+(-1)^k (k+1)}{m_{k-1}}x_k\mod J_W*_qJ_W$ (one can check that the coefficient is always an integer).

For the $n$th power, Lemma \ref{simple} gives
\begin{equation*}\label{nuWmod}
[n]_W(u)=nu+\sum \limits_{k \ge 1} r_k \frac{(nu)^{k+1}-nu^{k+1}}{m_k} \mod {J_W *_q J_W}.
\end{equation*}

Combining this with the formula \eqref{nuW}, we obtain that the coefficient for $u^{k+1}$ in the series $[n]_W(u)$ is
\begin{equation}\label{coeff}
\begin{cases}
\hspace{20ex} n &\text{for $k=0$,}\\
\hspace{17ex}-\frac{n(n-1)}{2}x_1 &\text{for $k=1$},\\
\hspace{1ex}\bigl ((4q+1)\frac{n(n-1)(n-2)}{6} -2q\frac{n(n-1)(n+1)}{3} \bigr)x_2 &\text{for $k=2$,}\\
-\frac{(n^{k+1}-n)}{m_k} \frac{(1+(-1)^k(k+1))}{m_{k-1}} x_k\mod{J_W*_qJ_W} &\text{for $k>2$.}
\end{cases}
\end{equation}

Note that substituting $k=1$ in the expression for $k>2$ in~\eqref{coeff} gives $\frac{n(n-1)}{2}x_1$, which coincides with the formula for $k=1$ up to a sign.

Now we can check the regularity of the sequence $(p, v_1, \ldots)$ from the statement of Landweber's theorem.

Fix a prime $p$. Then for the elements $v_n$, except for the case $v_1$ for $p=3$, we obtain from~\eqref{coeff} that 
\[
  v_n = \pm\frac{(p^{p^n}-p)}{p}\frac{(1\pm p^n)}{m_{p^n-2}} x_{p^n-1} \mod{J_W*_qJ_W} =\varepsilon_n x_{p^n-1} \mod{J_W*_qJ_W},
\]  
where $\varepsilon_n \not \equiv 0 \mod p$.
This implies the regularity of the sequence $(p, v_1, \ldots)$ for $p\ne 3$.

For the case of $p=3$, we note that $v_1$ is a coefficient for $u^3$ in $[3]_W(u)$, that is, $(4q+1 -16q)x_2 = (1-12q)x_2 \equiv x_2 \mod 3$.
Thus, the sequence $(3, v_1, \ldots)$ is also regular.
\end{proof}

Let $\P$ be the set of primes of the form $p=2^k+1$ (Fermat primes) greater than $3$. Consider the theory $W^*[\P^{-1}]$ obtained from $(W^*, *_q)$ by inverting all $p \in \P$.

\begin{theorem}\label{Fermat}
For the theory $W^*[\P^{-1}]$, there exists a complex orientation such that the coefficients of the corresponding formal group law generate the whole ring $(\varOmega_W^*[\P^{-1}], *_q)$ (as a $\Z[\P^{-1}]$-algebra).
\end{theorem}

Theorem \ref{Fermat} is stated in \cite{buch72} for the standard multiplication given by the Stong projection. Below we present its proof based on the results obtained in \cite{c-p23}.

\begin{proof}
We consider the theory $(W^*, *_q)$ with an arbitrary complex orientation. By Theorem~\ref{Wkring} we have $(\varOmega^*_W, *_q)=\Z[x_1, x_2, \ldots]/(x_1 *_q x_1 = (4q+1)x_2)$. By \cite[Lemma~3.8]{c-p23}, for the formal group law $\widetilde F_W(u, v) = u + v + \sum \widetilde \omega_{ij}*_q u^{*_qi} * v^{ *_qj}$ defined by a complex orientation of~$W^*$, we have $\widetilde \omega_{11} = (2 l +1)x_1$ for $l \in \Z$ and $\widetilde \omega_{12 } = 3w_2-2qx_2$ for $w_2 \in \varOmega^{-4}_W$, where $l$ and $w_2$ are defined by the complex orientation and can take arbitrary values. Therefore, choosing an orientation such that $l=0, \, w_2=0$, we get $x_1=\widetilde \omega_{11}, \, x_2 = \widetilde \omega_{11}*_q\widetilde \omega_{11} +2\widetilde \omega_{12}$.

It remains to prove that for some orientation the generators $x_k$, $k>2$, can be expressed via the coefficients of~$\widetilde F_W$. To do this, it is necessary and sufficient to show that there exists an integral linear combination of $\widetilde \omega_{ij}, \, i+j=k+1>3$, with the $s$-number equal to $s_k(x_k)=m_k m_{k-1 }$ up to elements invertible in $\Z[\P^{-1}]$. By \cite[Lemma~3.11]{c-p23} we have
\begin{equation}\label{gcd}
      \gcd \bigl\{ s_{i+j-1}(\widetilde \omega_{ij}) \mid i+j = k+1\bigr\}
  = m_k \bigl ( 1+(-1)^k (k+1)+ c_k m_km_{k-1} \bigr ), \quad k>2,
\end{equation}
where $c_k$ can be arbitrary integers depending on the orientation.

The same argument as in \cite[Lemma~3.11]{c-p23} shows that for formal group laws over the localized rings $W^*[\P^{-1}]$ the equality~\eqref{gcd} holds with $c_k \in \Z[\P^{-1}]$.

Thus, to prove the theorem \ref{Fermat} it is necessary and sufficient to prove that for every $k>2$ there exists $c_k \in \P$ such that $$m_k \bigl ( 1+(-1)^ k (k+1)+ c_k m_km_{k-1} \bigr)= \varepsilon_k m_k m_{k-1},$$ where $\varepsilon_k$ is an invertible element of the ring $ \Z [\P^{- 1}]$, or equivalently
\begin{equation}\label{=}
    1+(-1)^k (k+1)+ c_k m_km_{k-1} = \varepsilon_k m_{k-1}.
\end{equation}

Let us show that for $k \ne 8$ the number $c_k$ can be chosen so that equality~\eqref{=} holds with $\varepsilon_k=1$. In this case $$c_k = \frac{m_{k-1}-1 - (-1)^k(k+1)}{m_k m_{k-1}}$$ and we need to show that $c_k \in \Z[\P^{-1}]$.

If $m_{k-1}=1$ then $c_k=(-1)^{k+1}\frac{k+1}{m_k} \in \Z$ since $m_k$ always divides $ k+1$.

If $m_{k-1}=p$ is an odd prime, then $k=p^s$. Hence $m_k=1$ or~$2$. Then $c_k=\frac{p-1+p^s+1}{p m_k}=\frac{p^{s-1}+1}{m_k} \in \Z$, since $p^{ s-1}+1$ is even.

If $m_{k-1}=2$, then $k=2^{\ell}$ and $c_k = \frac{2-1-2^{\ell}-1}{2m_k} = -\frac{1}{m_k}2^{\ell-1}$. In this case, either $m_k=1$, in which case $c_k \in \Z$, or $m_k = p >2$. In the latter case, we have that $k+1=2^{\ell}+1=p^s$. Then either $k=8=2^3=3^2-1$ or $\ell=2^n$ and $s=1$ (see, for example, \cite[Lemma 3.15]{c-p23} ). In the latter case, we have $p=2^{2^n}+1$ is a Fermat prime, and $p=k+1>3$, and hence $p \in \P$. Then $c_k = -\frac{1}{p}2^{2^n-1} \in \Z[\P^{-1}]$.

It remains to consider the case $k=8$. Then \eqref{=} becomes $$1+9+ 6c_k = 2\varepsilon_k .$$ Setting $c_k=0$, we see that the equality holds for $\varepsilon_k = 5 \in \P$.
\end{proof}

\begin{remark}
In \cite[Theorem 3.13]{c-p23} it is proved that the ring ($\varOmega_W^*[1/2], *_q)$ is generated by the coefficients of the formal group law $F_W$ for some orientation of the theory $W^*$. On the other hand, the statement of Theorem \ref{Fermat} is only about orientations of the localized theory, which corresponds to non-integer values of $c_k$. The author doesn't know if it is possible to generate the localized ring $\varOmega_W^*[\P^{-1}]$ by the coefficients of the formal group law corresponding to the ``integral'' orientation of the theory $W^*$, that is, if there exist integers $c_k$ such that $1+(-1)^k(k+1)+c_k m_k m_{k-1}=\varepsilon_k m_{k-1}$ for invertible elements $\varepsilon_k \in \Z[\P^{-1}]$.
\end{remark}

\end{document}